\documentclass[11pt,twoside]{article}
\usepackage{latexsym,amsmath,amssymb}

\newif\ifdviwin

\usepackage{titlesec}
	
\titleformat{\section}{\large\bfseries\center}{\thesection}{1em}{\vspace{-.5cm}}
\titleformat{\subsection}[runin]{\bfseries}{\thesubsection}{1em}{}
\titleformat{\subsubsection}[runin]{\bfseries}{\thesubsubsection}{1em}{}
\usepackage[titles]{tocloft} 

\usepackage[english]{babel}
\usepackage{indentfirst}
\usepackage[mathscr]{eucal}
\usepackage{amssymb,amsmath,amsfonts}
\usepackage{fancybox,fancyhdr}
\usepackage{graphicx}
\usepackage[utf8]{inputenc}
\usepackage{float}
\usepackage{color}

\newif\ifdviwin

\dviwintrue

 \newtheorem{defi}{Definition}[section]
 \newtheorem{teo}[defi]{Theorem}
 \newtheorem{pro}[defi]{Proposition}
 \newtheorem{cor}[defi]{Corollary}

 \newenvironment{proof}{\rm \trivlist \item[\hskip \labelsep{\it
      Proof}.]}{\nopagebreak \hfill $\Box$ \endtrivlist}

\numberwithin{equation}{section}

\def\l3{\mathbb{L}^3}
\def\r{\mathbb{R}}
\def\a{\textbf{a}}
\def\v{\textbf{v}}

\begin{document}

\thispagestyle{empty}

\begin{center}

\renewcommand{\thefootnote}{\,}
{\fontsize{16}{19} \textbf{Invariant $\lambda$-translators in Lorentz-Minkowski space} 
\footnote{
\hspace{.15cm}\emph{Mathematics Subject Classification:} 53A10, 53C42, 34C05, 34C40}}

\vspace{0.5cm} { Antonio Bueno, Irene Ortiz}\\
\end{center}


\begin{abstract}
Given $\lambda\in\mathbb{R}$ and $\v\in\mathbb{L}^3$, a $\lambda$-translator with velocity $\v$ is an immersed surface in $\mathbb{L}^3$ whose mean curvature satisfies $H=\langle N,\v\rangle+\lambda$, where $N$ is a unit normal vector field. When $\lambda=0$, we fall into the class of translating solitons of the mean curvature flow. In this paper we study $\lambda$-translators in $\mathbb{L}^3$ that are invariant under a 1-parameter group of translations and rotations. The former are cylindrical surfaces and explicit parametrizations are found, distinguishing on the causality of both the ruling direction and the $\lambda$-translators. In the case of rotational $\lambda$-translators we distinguish between spacelike and timelike rotations and exhibit the qualitative properties of rotational $\lambda$-translators by analyzing the non-linear autonomous system fulfilled by the coordinate functions of the generating curves.
\end{abstract}

\setcounter{tocdepth}{2}
\tableofcontents

\section{Introduction}

In the last decades, the study of translating solitons of the mean curvature flow (translators for short) in the Euclidean 3-space has revealed as an important and fruitful theory, becoming an active field of research that has culminated in some groundbreaking results. We refer the reader to \cite{CSS, Hui, HuSi, Ilm, LiMa, MSHS, SpXi} and references therein for an outline of the development of this theory. Translators are solutions to the mean curvature flow that evolve purely by translations in some direction $\v\in\r^3$. In this case, the mean curvature flow equation has the expression
$$
H=\v^\bot,
$$
where $(\cdot)^\bot$ represents the normal component. In the literature, the vector $\v$ is commonly known as the \emph{velocity vector}. Translators naturally appear as limit surfaces after a proper rescaling near type II singularities \cite{HuSi}, and are characterized as minimal surfaces in a certain conformal space \cite{Ilm} or as weighted minimal surfaces in a space endowed with weighted area functional \cite{Gro}. All these characterizations make readily available techniques coming from different areas such as analysis of PDE's, Geometric Analysis and Geometric Measure Theory, among others.

A class of surfaces that are a generalization to the translators are the \emph{$\lambda$-translators}, defined by the prescribed mean curvature problem
$$
H=\v^\bot+\lambda,\quad\lambda\in\r.
$$
Obviously, when $\lambda=0$ we recover the translators. These $\lambda$-translators are solutions to a mean curvature flow with a constant forcing term \cite{Hui} and have constant weighted mean curvature equal to $\lambda$ \cite{Gro} in a density space. However, this class of surfaces has not received as much attention as translators have. In this sense, we refer the reader to \cite{BuLo, BLO, BuOr, Lop1, Lop2} for recent advances on $\lambda$-translators. We remark the independent works of López \cite{Lop1} and the authors \cite{BuOr}, where they classified invariant $\lambda$-translators under Euclidean translations and rotations.

Besides the Euclidean space, translators have been widely studied in other Riemannian spaces such as the product spaces \cite{Bue,LiMa}, the Heisenberg and Solvable groups \cite{Pip1,Pip2}, the Lorentz-Minkowski space \cite{Kim} and more generally in a Lorentzian setting with submersion structure \cite{LaOr} or in generalised Robertson-Walker geometries \cite{ALO} equipped with a natural conformal Killing time-like vector field. Specifically, in \cite{Kim} the author classified rotational spacelike translators in the Lorentz-Minkowski space following a phase space analysis of the solutions of the ODE fulfilled by the coordinates of the profile curve of such a rotational translator, while in \cite{LaOr} the authors extended this study and also approached timelike translators.

Following the development of the theory of $\lambda$-translators in the Euclidean space as a natural generalization of the theory of translating solitons, our goal in the present paper is to start to develop the theory of $\lambda$-translators in the Lorentz-Minkowski space $\l3$, by studying invariant $\lambda$-translators under 1-parameter groups of isometries. Specifically, we address the study of cylindrical and rotational $\lambda$-translators. The main difference here with the Euclidean space is to take into account both the causality of the $\lambda$-translator and the causality of either the translation direction or the rotation axis. Nonetheless, the results obtained and the shape of the profile curves strongly resemble the ones already achieved in the aforementioned pioneer works \cite{Kim,LaOr}.

We detail the organization of the paper and highlight some of the main results.

In Section \ref{secbasicnotation}, we start by setting the notation and defining the 1-parameter groups of isometries that will play the main role in this paper. In Sec. \ref{sec2} we classify cylindrical $\lambda$-translators when the translating direction is spacelike and timelike, distinguishing also between the causality of the $\lambda$-translator. We provide explicit parametrizations of cylindrical $\lambda$-translators by integrating the resulting ODE that relates the mean curvature with the curvature of the base curve.

The main contents of this paper are compiled in Section \ref{secclasificacionrotacionales}, where we classify rotational $\lambda$-translators. First, we prove in Prop. \ref{propejeparalelodensity} that the velocity vector and the rotation axis must be parallel. In Sec. \ref{subsectexistenceradial} we prove Th. \ref{thexistenciaradial}, where we discuss and exhibit the existence of rotational $\lambda$-translators intersecting orthogonally the rotation axis. Then, in Sec. \ref{subsecclasificacionrotacionalestimelikeaxis} we classify rotational $\lambda$-translators about a timelike axis, while in Sec. \ref{subsecclasificacionrotacionalesspacelikeaxis} we assume the spacelike causality of the rotation axis. In both cases, we treat the cases where the $\lambda$-translator is spacelike and timelike.

\section{Preliminaries}\label{secbasicnotation}

Let us fix the basic notation that is used throughout this work. The Lorentz-Minkowski space $\l3$ is the space $\r^3$ with coordinates $(x,y,z)$, endowed with the metric $\langle\cdot,\cdot\rangle=dx^2+dy^2-dz^2$. A global basis is given by the unitary vectors
$$
e_1=(1,0,0),\quad e_2=(0,1,0),\quad e_3=(0,0,1).
$$
A vector $\v\in\l3$ is spacelike if $\langle \v,\v\rangle>0$, timelike if $\langle \v,\v\rangle<0$ and lightlike if $\langle \v,\v\rangle=0$. A vector subspace $V$ is spacelike (resp. timelike or lightlike) if $\langle\cdot,\cdot\rangle_{|V}$ is spacelike (resp. is timelike or lightlike). If $\textbf{n}_V$ denotes an orthogonal vector to $V$ in the Euclidean sense, then $V$ is spacelike (resp. timelike or lightlike) if and only if $\textbf{n}_V$ is timelike (resp. spacelike or lightlike).

An immersed surface $\Sigma$ in $\l3$ is spacelike (resp. timelike) if at every $p\in\Sigma$ the tangent plane $T_p\Sigma$ is spacelike (resp. timelike). As a matter of fact, if $N$ is a unit normal on $\Sigma$, then $\Sigma$ is spacelike (resp. timelike) if and only if $N$ is timelike (resp. spacelike).

In this paper we focus on the following class of immersed surfaces.
\begin{defi}\label{deflambdatrans}
Given $\lambda\in\r$ and $\v\in\l3$, a surface $\Sigma$ in $\l3$ is a \emph{$\lambda$-translator} if its mean curvature satisfies
	\begin{equation}\label{eqdeflambdatranslator}
		H(p)=\langle N(p),\v\rangle+\lambda,\quad p\in\Sigma,
	\end{equation}
	where $N$ is a unit normal vector field on $\Sigma$.
\end{defi}
Obviously, when $\lambda=0$ we have translating solitons of the mean curvature flow, or \emph{translators} for short. In similar fashion as for translators, the vector $v$ will be referred as the \emph{velocity vector}. We assume $\lambda\neq0$ since otherwise we are studying translating solitons. Furthermore, after a change of the orientation we assume $\lambda>0$ without losing generality.

Next we recall isometries of $\l3$ that will play a crucial role in this paper. On the one hand, we consider the translations in a fixed direction $\a\in\l3$ which are the 1-parameter family of isometries given by
$$
T_\a(p,t)=p+t\a,\quad p\in\l3,\ t\in\r.
$$
On the other hand, we also consider the rotations in $\l3$. They have a different behavior as in $\r^3$ depending on the causality of the rotation axis. In this paper we focus on rotations about timelike and spacelike axes.
\begin{enumerate}
	\item \textbf{The rotation axis is timelike.} After a change of coordinates, we assume that the rotation axis agrees with the $e_3$-axis. Then, any such rotation is represented by the map
	$$
	\mathcal{S}_\theta(x,y,z)=
	\left(
	\begin{matrix}
		\cos\theta&-\sin\theta&0\\
		\sin\theta&\cos\theta&0\\
		0&0&1
	\end{matrix}
	\right)
	\left(
	\begin{matrix}
	x\\y\\z
	\end{matrix}
	\right)
	$$
	\item \textbf{The rotation axis is spacelike.} After a change of coordinates, we assume that the rotation axis agrees with the $e_1$-axis. Then, any such rotation is represented by the map
	$$
	\mathcal{T}_\theta(x,y,z)=
	\left(
	\begin{matrix}
		1&0&0\\
		0&\cosh\theta&\sinh\theta\\
		0&\sinh\theta&\cosh\theta
	\end{matrix}
	\right)
	\left(
	\begin{matrix}
	x\\y\\z
	\end{matrix}
	\right)
	$$
\end{enumerate}

\section{Cylindrical $\lambda$-translators in $\l3$}\label{sec2}

Given $\a\in\l3$, a cylindrical surface $\Sigma$ in the $\a$-direction is a surface invariant by the 1-parameter group of translations $T_\a$, i.e. such that $T_\a(p,t)\in\Sigma$ for every $p\in\Sigma$ and $t\in\r$. Every cylindrical surface can be suitably parametrized in terms of its \emph{base curve}, which is defined as the unique curve $\alpha(s)$ in a plane orthogonal\footnote{This orthogonallity has to be understood in the Euclidean sense} to $\a$ such that $\Sigma=T_\a(\alpha(s),t)$. As a matter of fact, a parametrization of a cylindrical surface with base curve $\alpha$ and rulings in the direction $\a$ is
$$
\psi(s,t)=\alpha(s)+t\a,\quad s\in I,t\in\r.
$$

Note that the immersion $\psi$ is flat and $2H=\kappa_\alpha$, where $\kappa_\alpha$ is the curvature of $\alpha$ as a planar curve. If we impose moreover $\Sigma$ to be a $\lambda$-translator, then $H=\langle\textbf{n}_\alpha,v\rangle+\lambda$, where $\textbf{n}_\alpha$ is a unit normal for $\alpha$. This yields a prescribed relation between the coordinates of $\alpha$ and its curvature $\kappa_\alpha$, in terms of $v$ and $\lambda$.

Next we distinguish between the causality of the ruling direction $a$. In this paper we only focus when the ruling direction $\a$ is either timelike or spacelike.
\begin{enumerate}
	
	\item[1.] If $\a$ is timelike, after a change of coordinates we assume $\a=e_3$, hence $\alpha(s)=(x(s),y(s),0)$. 
	\begin{enumerate}
		\item[1.1.] If the velocity vector $\v$ is parallel to $e_3$, then $\kappa_\alpha=2\lambda$ and we conclude that $\alpha$ is a circle. Hence $\Sigma$ is a vertical round cylinder. 
		\item[1.2.] On the contrary, after a rotation about the $e_3$-axis we assume $\v=(0,v_2,v_3),\ v_2>0$. Hence, $\langle\textbf{n}_\alpha,\v\rangle=v_2x'(s)$. If $\alpha$ is arc-length parametrized then $x'(s)=\cos\theta(s)$ and $y'(s)=\sin\theta(s)$, which yields
		$$
		\theta'=2(v_2\cos\theta(s)+\lambda).
		$$
		This ODE already appeared in both \cite{BuOr,Lop1}, where the solutions were explicitly integrated, hence details are skipped.
	\end{enumerate}

	\item[2.] If $\a$ is spacelike, after a change of coordinates we assume $\a=e_2$, hence a cylindrical surface $\Sigma$ is parametrized by
	$$
	\psi(s,t)=\alpha(s)+te_2=(x(s),t,z(s)),\qquad \textbf{n}_\alpha=(z'(s),0,x'(s)).
	$$
	At this point, we must also distinguish between the causality of $\Sigma$. 
	\begin{enumerate}
		\item[2.1.] If $\Sigma=\psi(s,t)$ is spacelike, then $\textbf{n}_\alpha$ is timelike and assuming moreover that $\alpha$ is arc-length parametrized yields $z'(s)^2-x'(s)^2=-1$, which leads to $x'(s)=\cosh\theta(s),\ z'(s)=\sinh\theta(s)$. After an ambient isometry we can also suppose $\v=(0,v_2,v_3),\ v_3>0$, hence $\langle\textbf{n}_\alpha,\v\rangle=-v_3\cosh\theta(s)$. From now on, we omit the dependence on the variable $s$. 
		
		Since the immersion $\psi$ is flat, its mean curvature $H$ and the curvature $\kappa_\alpha$ of $\alpha$ are related by the ODE 
		\begin{equation}\label{eqSS}
			\theta'=2(v_3\cosh\theta-\lambda).
		\end{equation} 
		
		\item[2.2.] If $\Sigma=\psi(s,t)$ is timelike, then $\textbf{n}_\alpha$ is spacelike, hence $z'^2-x'^2=1$ and consequently $x'=\sinh\theta,z'=\cosh\theta$. This time, the mean curvature $H$ and curvature $\kappa_\alpha$ of $\alpha$ are related by the ODE 
		\begin{equation}\label{eqST}
			\theta'=2(\lambda-v_3\sinh\theta).
		\end{equation}
	\end{enumerate}
\end{enumerate}

Now, we can derive the following results by explicit integration in \eqref{eqSS} and \eqref{eqST}, respectively, bearing in mind the usual change of variable $u=\tanh(\theta/2)$ that gives the relations $\cosh\theta=(1+u^2)/(1-u^2)$ and $\sinh\theta=2u/(1-u^2)$. 

\begin{teo}\label{thclasificacioncilindricasspacelike}
A spacelike cylindrical $\lambda$-translator along a spacelike direction can be parametrized, up to a change of coordinates, as follows:
\begin{itemize}
\item Case $\lambda>v_3$.
There is the straight line for $\theta_0$ such that $\cosh\theta_0=\lambda/v_3$, and
\begin{align*}
&\theta(s)=-2\mathrm{arctanh}\left(\sqrt{\frac{\lambda-v_3}{\lambda+v_3}}\mathrm{coth}(s\sqrt{\lambda^2-v_3^2})\right),\\
& x(s)=\frac{\lambda}{v_3} s-\frac{1}{v_3}\mathrm{arctanh}\left(\sqrt{\frac{\lambda+v_3}{\lambda-v_3}}\tanh\left(s\sqrt{\lambda^2-v_3^2}\right)\right),\\
& z(s)=-\frac{1}{2v_3}\log\left(1-\frac{v_3}{\lambda}\cosh\left(2s\sqrt{\lambda^2-v_3^2}\right)\right).
\end{align*}
See Fig. \ref{figcylindricalspacelikelambdamayorv3}.

\begin{figure}[ht]
\begin{center}
\includegraphics[width=.4\textwidth]{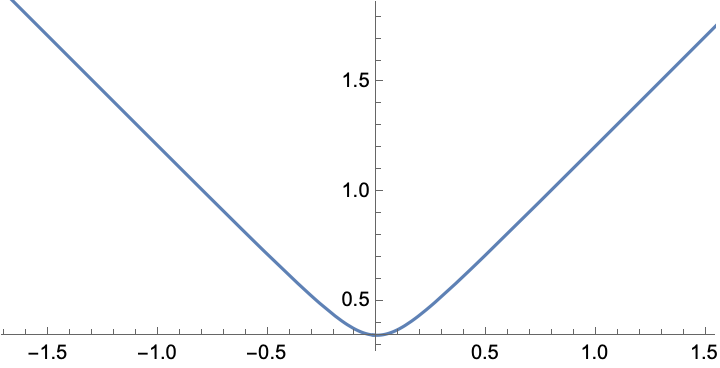}
\end{center}
\caption{The base curve of a cylindrical spacelike $\lambda$-translator for $\lambda=2$. Here, $v_3=1$.}
\label{figcylindricalspacelikelambdamayorv3}
\end{figure}

\item Case $\lambda=v_3$. 
There is the horizontal line for $\theta_0=0$, and
\begin{align*}
&\theta(s)=-2\mathrm{arccoth}(2v_3s),\\
&x(s)=s-\frac{1}{v_3}\mathrm{arctanh}(2v_3s),\\
&z(s)=-\frac{1}{2v_3}\log|1-4v_3^2s^2|.
\end{align*}
See Fig. \ref{figcylindricalspacelikelambdaigualv3}.

\begin{figure}[ht]
\begin{center}
\includegraphics[width=.4\textwidth]{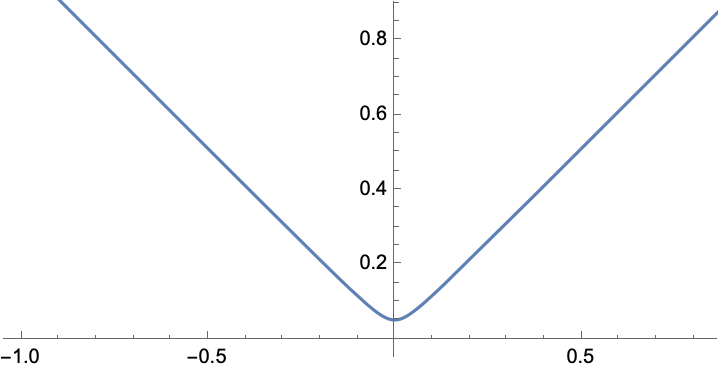}
\end{center}
\caption{The base curve of a cylindrical spacelike $\lambda$-translator for $\lambda=1$. Here, $v_3=1$.}
\label{figcylindricalspacelikelambdaigualv3}
\end{figure}

\item Case $\lambda<v_3$.
\begin{align*}
&\theta(s)=2\mathrm{arctanh}\left(\sqrt{\frac{v_3-\lambda}{v_3+\lambda}}\tan(s\sqrt{v_3^2-\lambda^2})\right),\\
&x(s)=\frac{\lambda}{v_3} s+\frac{1}{v_3}\mathrm{arctanh}\left(\sqrt{\frac{v_3-\lambda}{v_3+\lambda}}\tan\left(s\sqrt{v_3^2-\lambda^2}\right)\right),\\
&z(s)=-\frac{1}{2v_3}\log\left(1+\frac{v_3}{\lambda}\cos\left(2s\sqrt{v_3^2-\lambda^2}\right)\right).
\end{align*}
See Fig. \ref{figcylindricalspacelikelambdamenorv3}.

\begin{figure}[h]
\begin{center}
\includegraphics[width=.5\textwidth]{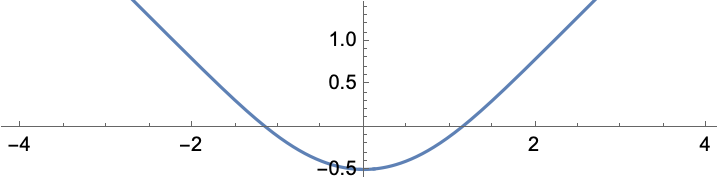}
\end{center}
\caption{The base curve of a cylindrical spacelike $\lambda$-translator for $\lambda=0'6$. Here, $v_3=1$.}
\label{figcylindricalspacelikelambdamenorv3}
\end{figure}

\end{itemize}
\end{teo}

For timelike $\lambda$-translators, we obtain the following.
\begin{teo}
A timelike cylindrical $\lambda$-translator along a spacelike direction can be parametrized, up to a change of coordinates, as follows:

There is the straight line for $\theta_0$ such that $\sinh\theta_0=\lambda/v_3$, and

\begin{itemize}
\item If $\lambda\neq v_3$, then
\begin{align*}
&\theta(s)=2\mathrm{arctanh}\left(\frac{-v_3+\sqrt{v_3^2+\lambda^2}\tanh(s\sqrt{v_3^2+\lambda^2})}{\lambda}\right),\\
&x(s)=\frac{\lambda}{v_3} s+\frac{1}{2v_3}\log\frac{1-\displaystyle{\frac{\sqrt{\lambda^2+v_3^2}\tanh(s\sqrt{\lambda^2+v_3^2)}}{v_3+\lambda}}}{1+\displaystyle{\frac{\sqrt{\lambda^2+v_3^2}\tanh(s\sqrt{\lambda^2+v_3^2)}}{\lambda-v_3}}},\\
&z(s)=\frac{1}{2v_3}\log\left(\frac{\lambda(v_3^2+\lambda^2)}{\lambda^2-v_3\left(v_3\cosh(2s\sqrt{v_3^2+\lambda^2})+\sqrt{v_3^2+\lambda^2}\sinh(2s\sqrt{v_3^2+\lambda^2})\right)}\right).
\end{align*}
See Fig. \ref{figcylindricaltimelike}, left.

\item If $\lambda=v_3$, then
\begin{align*}
&\theta(s)=-2\mathrm{arctanh}\left(1-\sqrt{2}\tanh(\sqrt{2}\lambda s)\right),\\
&x(s)=s+\frac{1}{\lambda}\mathrm{arctanh}\left(1-\sqrt{2}\tanh(\sqrt{2}\lambda s)\right),\\
&z(s)=\frac{1}{\lambda}\left(-\mathrm{arctanh}(1-2\tanh(\sqrt{2}\lambda s)+\mathrm{arctanh}\left((-1+\sqrt{2})(-1+\sqrt{2}-2\tanh(\sqrt{2}\lambda s)\right)\right).
\end{align*}
See Fig. \ref{figcylindricaltimelike}, right.
\end{itemize}
\end{teo}

\begin{figure}[h]
\begin{center}
\includegraphics[width=.2\textwidth]{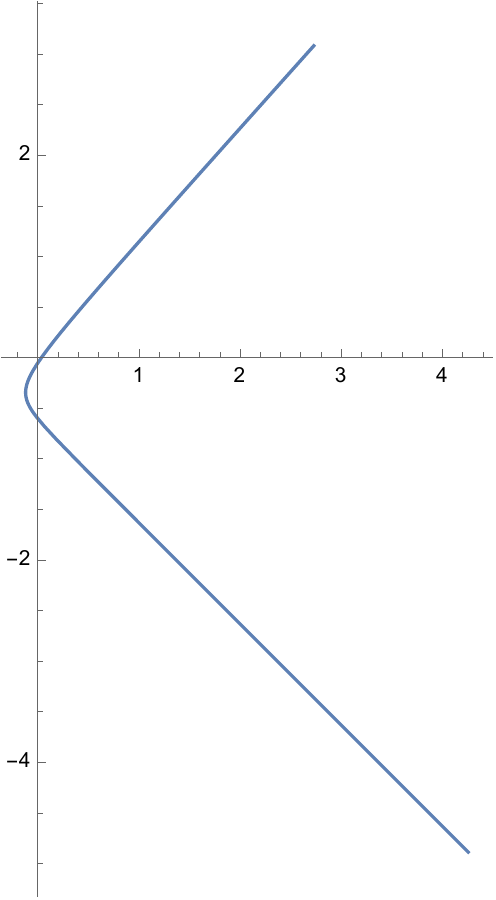}\hspace{1cm}
\includegraphics[width=.15\textwidth]{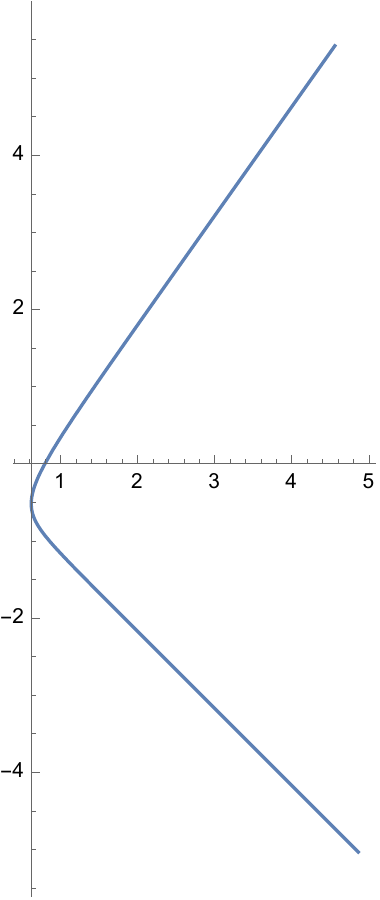}
\end{center}
\caption{Left: the base curve of a cylindrical timelike $\lambda$-translator for $\lambda=2$. Right: for $\lambda=1$. In both cases, $v_3=1$.}
\label{figcylindricaltimelike}
\end{figure}
\section{Rotational $\lambda$-translators}\label{secclasificacionrotacionales}
In this section, we classify rotational surfaces in $\l3$ about a timelike and spacelike axis that are $\lambda$-translators. When $\lambda=0$, i.e. for translating solitons of the mean curvature flow, this classification result was achieved in \cite{Kim, LaOr}. When fixing a causality for the rotation axis, we have to also take into account the causality of the surface, facing four different scenarios.

We remark that in the previous section, where we have classified cylindrical $\lambda$-translators, we have not assumed an a priori relation between the velocity vector $\v$ and the ruling direction $\a$. However, when approaching the study of rotational $\lambda$-translators we prove that $\v$ and the rotation axis must be parallel.

\begin{pro}\label{propejeparalelodensity}
Let $\Sigma$ be a $\lambda$-translator with velocity vector $\v$ that is rotational about some axis $L$. Then, $\v$ and $L$ must be parallel, or $\Sigma$ is a plane orthogonal to $L$.
\end{pro}

\begin{proof}
We only exhibit the proof when $L$ is timelike, since the proof in the spacelike case is similar. A rotational $\lambda$-translator about the $e_3$-axis is locally parametrized as
$$
\psi(s,t)=(x(s)\cos t,x(s)\sin t,z(s)),\quad s\in I, t\in\r,
$$
where $\alpha(s)=(x(s),0,z(s))$ is the profile curve, which is supposed to be arc-length parametrized. Hence $z'(s)^2-x'(s)^2=\epsilon$, where $\epsilon=-1$ if $\Sigma$ is spacelike and $\epsilon=1$ if $\Sigma$ is timelike. Assume that $\epsilon=-1$ since the opposite case follows the same. Consequently, $x'(s)=\cosh\theta(s)$ and $z'(s)=\sinh\theta(s)$ for some function $\theta$. Then, Eq. \eqref{eqdeflambdatranslator} yields
$$
-\theta'(s)-\frac{\sinh\theta(s)}{x(s)}=2\lambda+2v_1\sinh\theta(s)\sin t+2v_2\sinh\theta(s)\cos t-2v_3\cosh\theta(s).
$$
By the linear independence of the functions $\cos t,\sin t$, we conclude
$$
v_1\sinh\theta(s)=v_2\sinh\theta(s)=0
$$
for all $s\in I$. Now, if $\theta(s_0)\neq0$ for some $s_0\in I$ then $v_1=v_2=0$ and $v$ is parallel to $e_3$. If $\sinh\theta(s)$ is identically zero then $\theta(s)$ vanishes, $\alpha(s)$ is a horizontal line and $\Sigma$ is a horizontal plane, which is a $\lambda$-translator for the value $\lambda=v_3$.

The proof when $L$ is spacelike is similar, but this time concluding with the linear independence of the functions $\cosh t$ and $\sinh t$. Details are skipped.
\end{proof}

As a consequence, when rotating about a timelike axis we will always assume that the axis $L$ is the one generated by the vector $e_3$, hence $\v=e_3$. On the contrary, if the rotation axis is spacelike we will assume $L$ to be generated by $e_1$, hence $\v=e_1$.

\subsection{Existence of radial $\lambda$-translators}\label{subsectexistenceradial}

One of our interests is the study of rotational $\lambda$-translators with a particular geometric feature, as for example when the surface meets orthogonally the rotation axis. Recall that in the proof of Prop. \ref{propejeparalelodensity} we parametrized a spacelike, rotational $\lambda$-translator about the $e_3$-axis in terms of the coordinates of its profile curve, exhibiting that it is a solution to the ODE system
\begin{equation}\label{eqsys}
\left\lbrace
\begin{array}{l}
x'(s)=\cosh\theta(s),\\
z'(s)=\sinh\theta(s),\\
\theta'(s)=2(\cosh\theta(s)-\lambda)-\displaystyle{\frac{\sinh\theta(s)}{x(s)}}.
\end{array}\right.
\end{equation}
Conversely, every solution to this system corresponds to the profile curve of a spacelike, rotational $\lambda$-translator about the $e_3$-axis. 

The standard theory yields existence and uniqueness of \eqref{eqsys} when $x(s_0)>0$. Nonetheless, if $x(s_0)=0$ the third equation presents a singularity and thus the existence is not a direct consequence of the corresponding Cauchy problem. The objective of this section is to study this case when also $\theta(s_0)=0$, which would correspond to a rotational $\lambda$-translator intersecting orthogonally the $e_3$-axis.

We start with the case that the axis of rotation is timelike, which is assumed to be the $e_3$-axis after a change of coordinates. If we seek an orthogonal intersection with the $e_3$-axis, then it only makes sense to study spacelike $\lambda$-translators. Locally, a solution to \eqref{eqsys} can be parametrized as a \emph{radial} surface, which is described by a function $u:[0,R]\rightarrow\r$ and the parametrization
$$
\psi(r,\theta)=(r\cos\theta,r\sin\theta,u(r)),\quad r\in[0,R],\ \theta\in\r.
$$
A unit normal is
$$
N=\frac{1}{\sqrt{1-u'(r)^2}}(u'(r)\cos\theta,u'(r)\sin\theta,1),
$$
while a straightforward computation yields
\begin{equation}\label{eqHradial}
2H=-\frac{u''(r)}{(1-u'(r)^2)^{3/2}}-\frac{u'(r)}{r\sqrt{1-u'(r)^2}},
\end{equation}
hence any radial $\lambda$-translator is a solution to
$$
-\frac{u''(r)}{(1-u'(r)^2)^{3/2}}-\frac{u'(r)}{r\sqrt{1-u'(r)^2}}=2\left(-\frac{1}{\sqrt{1-u'(r)^2}}+\lambda\right).
$$
The existence of radial solutions to \eqref{eqsys} will follow from the following result.
\begin{teo}\label{thexistenciaradial}
There exists some $R>0$ such that there exists a unique solution to the initial value problem
\begin{equation}\label{dirichletradial}
\left\lbrace
\begin{array}{lll}
\displaystyle{\frac{u''(r)}{(1-u'(r)^2)^{3/2}}+\frac{u'(r)}{r\sqrt{1-u'(r)^2}}=2\left(\frac{1}{\sqrt{1-u'(r)^2}}-\lambda\right)}& \mathrm{in}&(0,R),\\
u(0)=u'(0)=0.
\end{array}
\right.
\end{equation}
In particular, there exists a unique spacelike $\lambda$-translator intersecting orthogonally the $e_3$-axis.
\end{teo}

\begin{proof}
After multiplying Eq. \eqref{eqHradial} by $r$, we obtain
$$
-2Hr=\frac{ru''(r)}{(1-u'(r)^2)^{3/2}}+\frac{u'(r)}{\sqrt{1-u'(r)^2}}=\left(\frac{r u'(r)}{\sqrt{1-u'(r)^2}}\right)'.
$$
Let us define
\begin{align*}
\phi(y)=\frac{1}{\sqrt{1-y^2}}-\lambda, \qquad f(y)=\frac{y}{\sqrt{1-y^2}},
\end{align*}
and recall that the inverse of $f$ is $f^{-1}(y)=y/\sqrt{1+y^2}$. Therefore, we rewrite \eqref{dirichletradial} as
$$
(rf(u'(r))'=2r\phi(u'(r)),
$$
and solving $u(r)$ we arrive to
$$
u(r)=\int_0^rf^{-1}\left(\frac{2}{t}\int_0^t s\phi(u'(s))ds\right)dt.
$$
This motivates the definition of the functional $\mathsf{T}:(C^1([0,R]),||.||)\rightarrow(C^1([0,R]),||.||)$, where we consider the usual norm $||u||=||u||_\infty+||u'||_\infty$, as follows
$$
(\mathsf{T}u)(r)=\int_0^rf^{-1}\left(\frac{2}{t}\int_0^t s\phi(u'(s))ds\right)dt.
$$
Let us point out that $u$ solves Eq. \eqref{dirichletradial} if and only if $u$ is a fixed point of $\mathsf{T}$. Our objective is to determine $R$ small enough such that $\mathsf{T}$ is a contraction in the closed ball $\overline{\mathcal{B}(0,R)}\subset(C^1([0,R]),||.||)$. Then, we conclude by applying the Banach fixed point theorem to conclude the existence of a fixed point of $\mathsf{T}$, hence of a solution to \eqref{dirichletradial}.

First, we prove that $\mathsf{T}$ is a self-map in a closed ball $\overline{\mathcal{B}(0,\epsilon)}$ for some $\epsilon$, i.e. $\mathsf{T}(\overline{\mathcal{B}(0,\epsilon)})\subset\overline{\mathcal{B}(0,\epsilon)}$. Let $u\in\overline{\mathcal{B}(0,\epsilon)}$. By taking $\epsilon<1/2$, the L'Hôpital rule yields
$$
\lim_{t\rightarrow0}\frac{2}{t}\int_0^ts\phi(u'(s))ds=0,
$$
since $||u||<\epsilon$ and in particular $u'(r)<1$. Then,
\begin{align*}
|(\mathsf{T}u)(r)|&\leq\int_0^r|f^{-1}\left(\frac{2}{t}\int_0^t s\phi(u'(s))ds\right)dt|\\
& \leq L_{f^{-1}}\int_0^r\frac{2}{t}\int_0^t |s\phi(u'(s))|dsdt\\
& \leq L_{f^{-1}}|\phi(0)|\int_0^r\frac{2}{t}\int_0^ts\,dsdt\\
& =L_{f^{-1}}|\phi(0)|\frac{r^2}{2}\leq L_{f^{-1}}|\phi(0)|\frac{R^2}{2},
\end{align*}
where $L_{f^{-1}}$ stands for the Lipschitz constants of $f^{-1}$. We can bound $|(\mathsf{T}u)'(r)|$ in a similar fashion, obtaining $|(\mathsf{T}u)'(r)|\leq L_{f^{-1}}\phi(0)R$. Therefore, by choosing $R$ small enough such that
$$
L_{f^{-1}}|\phi(0)|R\left(\frac{R}{2}+1\right)<\epsilon,
$$
we conclude that $\mathsf{T}(\overline{\mathcal{B}(0,\epsilon)})\subset\overline{\mathcal{B}(0,\epsilon)}$. Next, we prove that $\mathsf{T}$ is a contraction.
\begin{align*}
 |(\mathsf{T}u)(r)-(\mathsf{T}v)(r)| &\leq \int_0^r\Bigg|f^{-1}\left(\frac{2}{t}\int_0^t s\phi(u'(s))ds\right)-f^{-1}\left(\frac{2}{t}\int_0^t s\phi(v'(s))ds\right)\Bigg|dt\\
& \leq  2 L_{f^{-1}}\int_0^r\frac{1}{t}\left(\int_0^ts|\phi(u'(s))-\phi(v'(s))|ds\right)dt\\
& \leq 2 L_{f^{-1}}\int_0^r\frac{L_\phi}{t}\left(\int_0^ts|u'(s)-v'(s)|ds\right)dt\\
& \leq 2 L_{f^{-1}}L_\phi||u-v||\int_0^r\frac{1}{t}\left(\int_0^t sds\right)dt\\
& = L_{f^{-1}}L_\phi||u-v||\frac{r^2}{2}.
\end{align*}
Hence,
$$
||Tu-Tv||_\infty\leq L_{f^{-1}}L_\phi||u-v||\frac{r^2}{2}.
$$
Let $R_1$ be small enough such that $K_1=L_{f^{-1}}L_\phi\frac{r^2}{2}<\frac{1}{2}$ for every $r<R_1$. A similar argument allows us to bound $ |(\mathsf{T}u)'(r)-(\mathsf{T}v)'(r)|$ by
$$
|(\mathsf{T}u)'(r)-(\mathsf{T}v)'(r)|\leq L_{f^{-1}}L_\phi||u-v||r.
$$
Therefore,
$$
||(Tu)'-(Tv)'||_\infty\leq L_{f^{-1}}L_\phi||u-v||r.
$$
Let $R_2$ be small enough such that $K_2=L_{f^{-1}}L_\phi r<\frac{1}{2}$. We define $R_0=\min\{R_1,R_2\}$. Consequently, for $r\in[0,R_0]$ the following bound holds
$$
||Tu-Tv||\leq K_1||u-v|+K_2||u-v||<||u-v||,
$$
which implies that $\mathsf{T}$ is a contraction. Finally, we prove that $u$ extends with $C^2$-regularity at $r=0$. Taking limits as $r\rightarrow0$ and by the L'Hôpital rule, $-1+\lambda=-2u''(0)$, which yields $u''(0)=\frac{1-\lambda}{2}$. 
\end{proof}

If the rotation axis is spacelike, for instance the $e_1$-axis, a radial surface is parametrized by
$$
\psi(r,t)=(u(r),r\sinh t,r\cosh t).
$$
A similar computation as in the timelike case yields the following equation involving the mean curvature
$$
2H=-\frac{u'}{r\sqrt{1-u'^2}}-\frac{u''}{(1-u'^2)^{3/2}}.
$$
The result obtained and the computations involved are the same, hence they are skipped.
\begin{cor}\label{corexistenciaradial}
There exists a unique timelike $\lambda$-translator intersecting orthogonally the $e_1$-axis.
\end{cor}

\subsection{Rotational $\lambda$-translators about a timelike axis}\label{subsecclasificacionrotacionalestimelikeaxis}
We begin the classification of rotational $\lambda$-translators by assuming that the rotation axis is timelike. After a change of coordinates, we assume the rotation axis to be the $e_3$-axis. Let $\alpha(s)=(x(s),0,z(s)),\ s\in I\subset\mathbb{R}$ be an arc-length parametrized curve. The map
$$
\psi(s,t)=(x(s)\cos t,x(s)\sin t,z(s)),
$$
parametrizes the rotational surface $\Sigma$ generated by $\alpha(s)$. A unit normal is
$$
N=(z'(s)\cos t,z'(s)\sin t,x'(s)).
$$
From now, we omit the dependence of the variable $s$. A straightforward computation yields that the principal curvatures are
\begin{equation}\label{eqcurvprinc}
\kappa_1=\frac{x''z'-x'z''}{x'^2-z'^2},\qquad \kappa_2=-\frac{z'}{x}.
\end{equation}
Note that $\langle N,N\rangle=z'^2-x'^2=\epsilon,\epsilon=\pm1$, depending on the causality of the surface. That is, if $\Sigma$ is spacelike then $N$ is timelike and $\epsilon=-1$; analogously, if $\Sigma$ is timelike then $N$ is spacelike and $\epsilon=1$. 

\subsubsection{Spacelike $\lambda$-translators}
We begin by assuming that $\Sigma$ is spacelike, i.e. $z'^2-x'^2=-1$. Hence, there exists a function $\theta=\theta(s)$ such that $x'=\cosh\theta$ and $z'=\sinh\theta$. Since $\Sigma$ is a $\lambda$-translator, $H=\langle N,e_3\rangle+\lambda=-\cosh\theta+\lambda$, and from Eq. \eqref{eqcurvprinc} the following system is fulfilled
\begin{equation}\label{sysODE}
\left\lbrace
\begin{array}{l}
x'=\cosh\theta\\
\theta'=2(\cosh\theta-\lambda)-\displaystyle{\frac{\sinh\theta}{x}}.
\end{array}\right.
\end{equation}
We define the phase plane $\Theta=(0,\infty)\times\mathbb{R}$ as the space of solutions of \eqref{sysODE}, with coordinates $(x,\theta)$. The orbits $\gamma(s)=(x(s),\theta(s))$ are the solutions of \eqref{sysODE} when regarded in $\Theta$. The existence and uniqueness of the Cauchy problem of \eqref{sysODE} for an initial data $(x_0,\theta_0)\in\Theta$ ensures that the orbits provide a foliation of $\Theta$, and that two different orbits cannot intersect in $\Theta$.

We define $\Gamma=\Theta\cap\{x=\Gamma(\theta)\}$, where $\Gamma(\theta)$ is given by
$$
\Gamma(\theta)=\frac{\sinh\theta}{2(\cosh\theta-\lambda)},
$$
which is a (possibly disconnected) horizontal graph over the $\theta$-axis. When $\theta\rightarrow\infty$, $\Gamma(\theta)\rightarrow 1/2$ and $\Gamma$ converges to the line $x=1/2$. The points located at $\Gamma$ correspond to points of $\alpha$ with vanishing curvature, i.e. points for which $\theta'=0$. 

The curve $\Gamma$ divides $\Theta$ into connected components, called \emph{monotonicity regions}, where the coordinate functions of an orbit are monotonous. The structure of these monotonicity regions depends on three possible values for $\lambda$, which determine the behavior of $\Gamma$ and ultimately the structure of $\Theta$.

\textbf{\underline{Case $\lambda<1$}}

If $\lambda<1$, the denominator of $\Gamma(\theta)$ is always positive, hence $\Gamma$ is a connected graph over the $\theta$-axis that has $(0,0)$ as finite endpoint. Therefore, $\Theta$ is divided into two monotonicity regions, 
$$
\Theta_1=\{(x,\theta):\ x<\Gamma(\theta),\theta>0\},\quad \Theta_2=\{(x,\theta):\ x>\Gamma(\theta),\theta>0\}\cup\{(x,\theta):\ \theta<0\}.
$$
For instance, $x'>0$ everywhere, $\theta'>0$ in $\Theta_2$ and $\theta'<0$ in $\Theta_1$. See Fig \ref{fig1}, left.

\begin{figure}[h]
\begin{center}
\includegraphics[width=.4\textwidth]{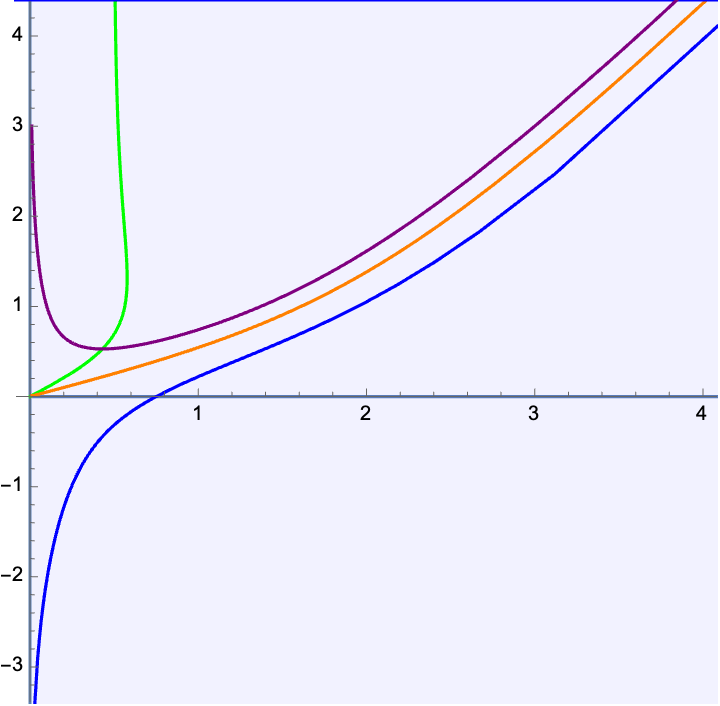}
\includegraphics[width=.55\textwidth]{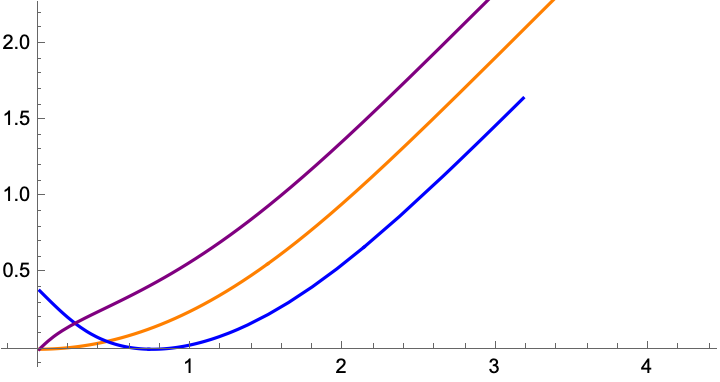} 
\end{center}
\caption{Left: the phase plane of a timelike rotational $\lambda$-translator about a timelike axis, and the possible orbits for $\lambda<1$. Right: the profile curves of the rotational $\lambda$-translators corresponding to each orbit. Here, $\lambda=1/2$.}
\label{fig1}
\end{figure}

Now we describe the structure of the orbits. First, from Th. \ref{thexistenciaradial} we know the existence of the orbit $\gamma_+$ having $(0,0)$ as endpoint. By monotonicity, $\gamma_+$ is contained in the region $\Theta_2$ with $\theta>0$. See Fig. \ref{fig1}, left, the orbit in orange. The corresponding rotational $\lambda$-soliton is a strictly convex, entire vertical graph that intersects orthogonally the rotation axis. See Fig. \ref{fig1}, right, the curve in orange.

Now, fix some $x_0>0$, let $\gamma_{x_0}(s)$ the orbit passing through $(x_0,0)$ at $s=0$ and $\Sigma_{x_0}$ the corresponding rotational $\lambda$-soliton. Recall that $z'=\sinh\theta$. When $s$ increases, $\gamma_{x_0}$ is strictly contained in $\Theta_2$ with $\theta>0$, hence $\Sigma_{x_0}$ has strictly increasing height function. When $s$ decreases, $\gamma_{x_0}(s)\rightarrow(0,\theta_0)$, $\theta_0<0$, as $s\rightarrow s_0<0$. Hence, for $s\in(s_0,0)$, $\Sigma_{x_0}$ has decreasing height function and reaches a minimum at $s=0$. Therefore, $\Sigma_{x_0}$, is a vertical graph with strictly positive Gauss curvature, non-monotonous height function, and converges to the axis of rotation with a cusp point where the derivative of its profile curve tends to $-\pi/4$ as $\theta_0\rightarrow-\infty$. See Fig \ref{fig1}, right, the curve in blue.

Finally, since the orbits foliate the whole $\Theta$, let $\gamma(s)$ be any orbit lying above $\gamma_+$. By monotonicity, $\gamma(s)$ must intersect the curve $\Gamma$, say at $s=0$. For $s\rightarrow\infty$, the orbit $\gamma(s)$ lies in $\Theta_2$ with $x(s),\theta(s)\rightarrow\infty$. When $s$ decreases, $\gamma(s)\rightarrow(0,\theta_1)$ as $s\rightarrow s_0<0$, and the profile curve converges to the axis of rotation with a cusp point where its derivative tends to $\pi/4$. Since $\theta(s)>0$ everywhere, $z'(s)>0$ and the height function of the corresponding rotational $\lambda$-soliton is strictly increasing. Furthermore, its Gauss curvature starts being negative, then vanishes and ends up being strictly positive everywhere. See Fig \ref{fig1}, right, the curve in purple.

\textbf{\underline{Case $\lambda=1$}}

If $\lambda=1$, the denominator of $\Gamma(\theta)$ vanishes at $\theta=0$. Therefore, $\Gamma$ has the $x$-axis as a horizontal asymptote, with $\Gamma(\theta)\rightarrow\infty$ as $\theta\rightarrow0$. 

\begin{figure}[h]
\begin{center}
\includegraphics[width=.45\textwidth]{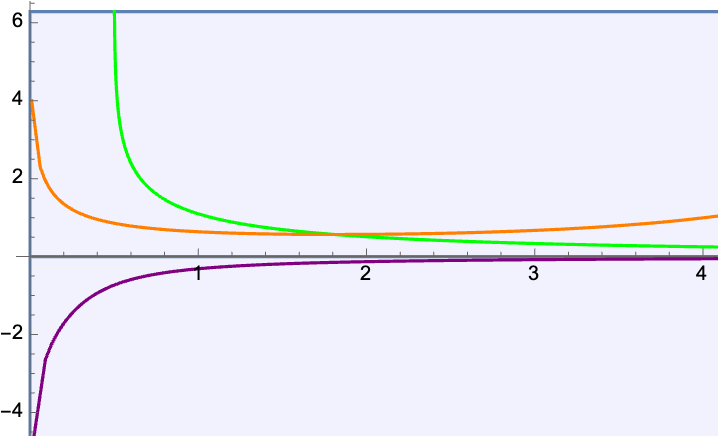}
\includegraphics[width=.45\textwidth]{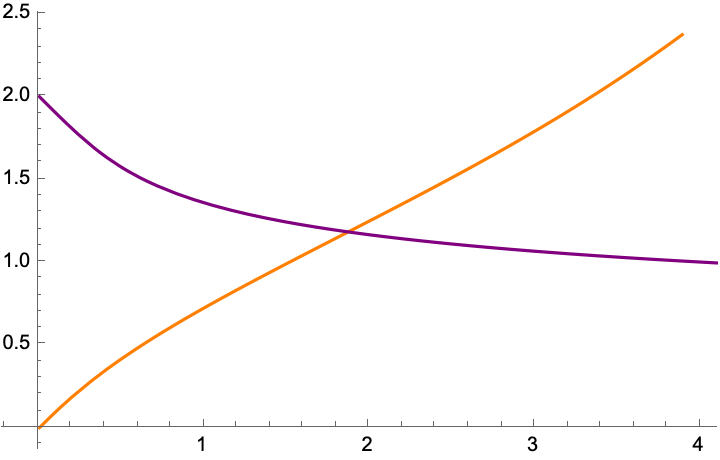} 
\end{center}
\caption{Left: the phase plane of a timelike rotational $\lambda$-translator about a timelike axis, and the possible orbits for $\lambda=1$. Right: the profile curves of the rotational $\lambda$-translators corresponding to each orbit.}
\label{fig2}
\end{figure}

First, we describe a trivial solution of \eqref{sysODE}. Since $\lambda=1$, the orbit $\gamma_+(s)=(s,0)$ solves \eqref{sysODE}. This orbit generates a horizontal plane, which is a $1$-soliton when oriented with unit normal $N=e_3$. By the uniqueness of the Cauchy problem, no orbit can intersect $\gamma_+$.

Now, assume that $\gamma(s)$ is an orbit lying in the monotonicity region $\{(x,\theta):\ \theta<0\}$. When $s\rightarrow\infty$, $x(s)\rightarrow\infty$ and $\theta(s)\rightarrow0$, i.e. $\gamma$ converges to $\gamma_+$. When $s$ decreases, it converges to some $s_0<0$ such that $\gamma(s)\rightarrow(0,\theta_0),\ \theta_0<0$. The corresponding rotational $\lambda$-soliton has strictly decreasing height function and negative Gauss curvature, and converges to the axis of rotation with a cusp point of angle tending to $-\pi/4$ as $\theta_0\rightarrow-\infty$.

Finally, assume that $\gamma(s)$ is an orbit lying in $\{(x,\theta):\ \theta>0\}$. Then, $\gamma(s)\rightarrow(0,\theta_1),\ \theta_1>0,$ as $s\rightarrow s_1<0$. On the other hand, $\gamma(s)$ intersects $\Gamma$ when $s$ increases and then $x(s),\theta(s)\rightarrow\infty$ as $s\rightarrow\infty$. This time the cusp point at the rotation axis tends to $\pi/4$ as $\theta_1\rightarrow\infty$.

%


\textbf{\underline{Case $\lambda>1$}}
Finally, assume that $\lambda>1$ and let $\theta_0>0$ be the unique positive solution to $\cosh\theta=\lambda$. Then, the denominator of $\Gamma(\theta)$ vanishes exactly at $\pm\theta_0$. Consequently, $\Gamma$ has the point $(0,0)$ as endpoint, appears in $\Theta$ for $\theta\in(-\theta_0,0]\cup(\theta_0,\infty)$ and has $\theta=\pm\theta_0$ as asymptotes.

\begin{figure}[h]
\begin{center}
\includegraphics[width=.45\textwidth]{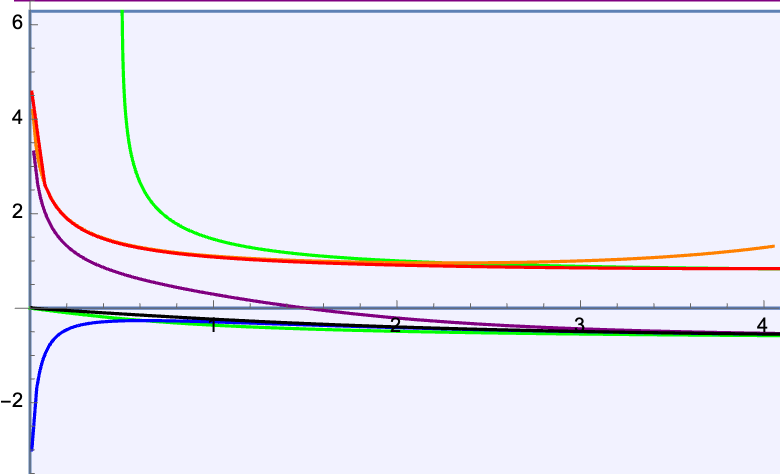}
\includegraphics[width=.45\textwidth]{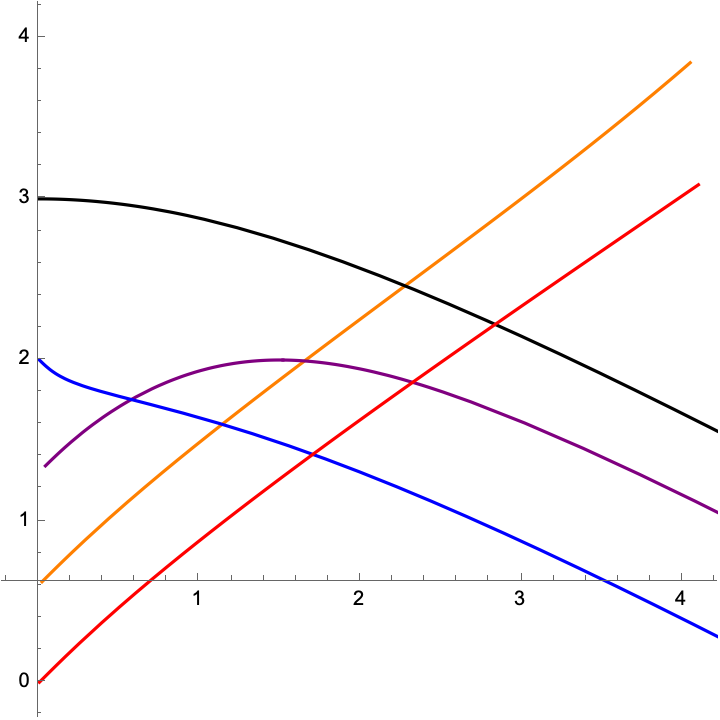} 
\end{center}
\caption{Left: the phase plane of a spacelike rotational $\lambda$-translator about a timelike axis, and the possible orbits for $\lambda>1$. Right: the profile curves of the rotational $\lambda$-translators corresponding to each orbit. Here, $\lambda=2$.}
\label{fig3}
\end{figure}

First, by Th. \ref{thexistenciaradial} there exists a rotational $\lambda$-translator intersecting orthogonally the rotation axis. By comparison, the height function at the intersection with the rotation axis is decreasing. Hence, its corresponding orbit has the point $(0,0)$ as endpoint, lies in the region $\theta<0,x>\Gamma(\theta)$ and converges to the line $\theta=-\theta_0$. This $\lambda$-translator has strictly positive Gauss curvature and strictly decreasing height function. See Fig. \ref{fig3}, the oribt and profile curve in black.

Let $\gamma$ be any orbit in the region $\{(x,\theta):\ \theta<0\}$. Then, $\gamma(s)$ converges to a finite point in the $\theta$-axis as $s$ decreases. When $s$ increases, $\gamma_1$ intersects $\Gamma$ and then ends up converging to the line $\theta=-\theta_0$. The corresponding rotational $\lambda$-translator has a cusp point at the rotation axis, strictly decreasing height function and Gauss curvature starting being negative and then being positive. See \ref{fig3}, the orbit and profile curve in blue.

Now, we fix some $x_0>0$ and let $\gamma(s)$ be the orbit passing through $(x_0,0)$ at $s=0$. When $s$ increases, $\gamma$ converges to the line $\theta=-\theta_0$. When $s$ decreases, $\gamma$ converges to a finite point, $(0,\theta_{x_0})$, in the $\theta$-axis. The corresponding rotational $\lambda$-translator has a cusp point at the rotation axis, strictly positive Gauss curvature and height function increasing and then decreasing. See Fig. \ref{fig3}, the orbit and profile curve in pruple.

Next, for each point $(x_1,\theta_1)\in\Gamma$ there exists an orbit $\gamma$ passing through $(x_1,\theta_1)$ at $s=0$. When $s\rightarrow\infty$, the coordinates of $\gamma$ satisfy $x(s),\theta(s)\rightarrow\infty$. When $s$ decreases, $\gamma$ converges to a finite point, $(0,\theta_{x_1})$, which obviously satisfies $\theta_{x_1}>\theta_1$. The rotational $\lambda$-translator has strictly increasing height function and Gauss curvature starting negative and then being positive. See Fig. \ref{fig3}, the orbit and curve in orange.

By uniqueness of the Cauchy problem, it must happen that $\theta_{x_0}$ and $\theta_{x_1}$ as previously defined satisfy $\theta_{x_0}\leq\theta_{x_1}$. Therefore, for $x_0\rightarrow\infty$ we have $\theta_{x_0}\nearrow\theta_\infty^0$. Similarly, for $x_1\rightarrow\infty$ we have $\theta_{x_1}\searrow\theta_\infty^1$. Clearly $\theta_\infty^0\leq\theta_\infty^1$.

Now, let be $\hat{\theta}\in[\theta_\infty^0,\theta_\infty^1]$ and let $\hat{\gamma}$ be the orbit having $(0,\hat{\theta})$ as endpoint. The only possibility for $\hat{\gamma}(s)$ is to converge to $\theta=\theta_0$ as $s\rightarrow\infty$. The corresponding rotational $\lambda$-translator has a cusp at the rotation axis, strictly negative Gauss curvature and strictly increasing height function. See Fig. \ref{fig3}, the orbit and profile curve in red.

This completes the classification of spacelike, rotational $\lambda$-translators about a timelike axis.

\subsubsection{Timelike $\lambda$-translators}
Now, we keep considering the rotation axis to be the $e_3$-axis, but this time the $\lambda$-translator is timelike. Hence, $N$ is spacelike and the coordinate functions of the profile curve satisfy $z'^2-x'^2=1$. This time, we ensure the existence of a function $\theta=\theta(s)$ such that $x'=\sinh\theta,\ z'=\cosh\theta$ and the following system is fulfilled
\begin{equation}
\left\lbrace
\begin{array}{l}
x'=\sinh\theta\\
\theta'=2(\sinh\theta-\lambda)-\displaystyle{\frac{\cosh\theta}{x}}.
\end{array}\right.
\end{equation}

First, we study the curve $\Gamma=\Theta\cap\{\Gamma(\theta)\}$, where this time $\Gamma(\theta)$ is
$$
\Gamma(\theta)=\frac{\cosh\theta}{2(\sinh\theta-\lambda)}.
$$
By naming $\theta_0$ to the unique solution to $\sinh\theta=\lambda$, we conclude that $\Gamma$ is a connected curve only defined for $\theta\in(\theta_0,\infty)$ that converges to the line $x=1/2$ as $\theta\rightarrow\infty$ and has the line $\theta=\theta_0$ as asymptote when $\theta\rightarrow\theta_0$. See Fig. \ref{fig4}, left, the curve in green. 

We begin our description of the orbits by fixing some $x_0>0$ and considering the orbit $\gamma$ passing through $(x_0,0)$ at $s=0$. When $s$ increases, $\gamma$ lies in the region $\{\theta<0\}$ and converges to some $(0,\theta_{x_0}),\ \theta_{x_0}^1<0$. When $s$ decreases, $\gamma$ lies in $\{\theta>0\}$ and converges to $(0,\theta_{x_0}^2),\ \theta_{x_0}^2>0$. The corresponding rotational $\lambda$-translator has strictly increasing height (since $z'=\cosh\theta$) and two cusp points at the rotation axis. See Fig. \ref{fig4}, the orbit and curve in purple.

Following the same ideas as developed in the case $\lambda>1$ in the spacelike setting, we conclude the existence of a 1-parameter family of orbits intersecting the curve $\Gamma$, and of the type that converges to $\theta=\theta_0$ without intersecting $\Gamma$. Details are skipped. See Fig. \ref{fig4}, the orbit and curve in orange and red, respectively.

\begin{figure}[h]
\begin{center}
\includegraphics[width=.45\textwidth]{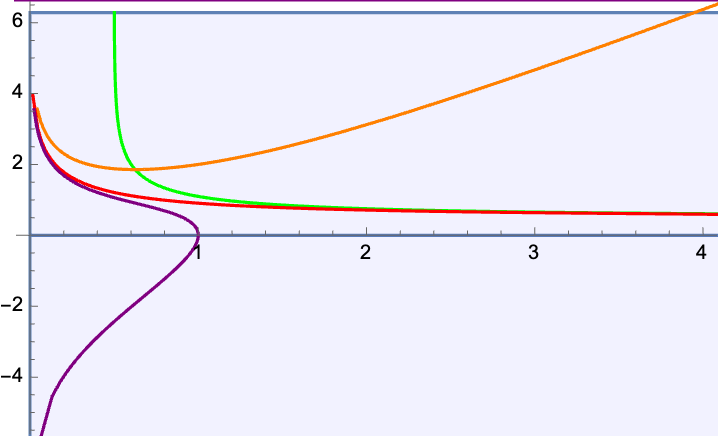}
\includegraphics[width=.35\textwidth]{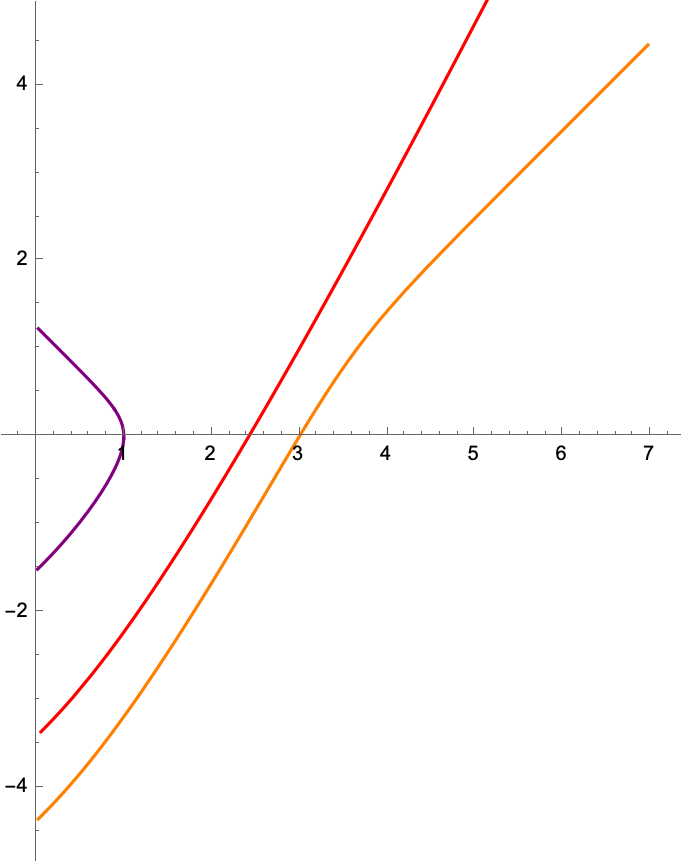} 
\end{center}
\caption{Left: the phase plane of a timelike rotational $\lambda$-translator about a timelike axis, and the possible orbits. Right: the profile curves of the rotational $\lambda$-translators corresponding to each orbit. Here, $\lambda=1/2$.}
\label{fig4}
\end{figure}

\subsection{Rotational $\lambda$-translators about a spacelike axis}\label{subsecclasificacionrotacionalesspacelikeaxis}

Once we have classified rotational $\lambda$-translators about a timelike axis, we consider the case where the rotation axis is spacelike. After a change of coordinates we assume the $e_1$-axis to be this rotation axis. Hence, we rotate an arc-length parametrized curve $(x(s),0,z(s))$ about the $e_1$-axis, obtaining a surface $\Sigma$ parametrized by
$$
\psi(s,t)=(x(s),z(s)\cosh t,z(s)\sinh t),\quad z(s)>0,
$$
whose unit normal is
$$
N=(z'(s),x'(s)\cosh t,x'(s)\sinh t).
$$
Recall that this time the angle function is measured by taking the velocity vector to be $\v=e_1$, i.e. $\nu=\langle N,e_1\rangle=z'(s)$. Again, from now on we omit the dependence on the variable $s$. In a similar way to the previous case, we distinguish the causality of the surface to be timelike or spacelike.

\subsubsection{The spacelike case}
If $\Sigma$ is spacelike then $N$ is timelike, i.e. $z'^2-x'^2=-1$, hence $z'=\sinh\theta$ and $x'=\cosh\theta$. The associated nonlinear autonomous system is
\begin{equation}
\left\lbrace
\begin{array}{l}
z'=\sinh\theta\\
\theta'=-2(\sinh\theta+\lambda)-\displaystyle{\frac{\cosh\theta}{z}}.
\end{array}\right.
\end{equation}
We start by determining the structure of the phase plane by the behavior of the graph $z=\Gamma(\theta)$, which has the expression
$$
\Gamma(\theta)=-\frac{\cosh\theta}{2(\sinh\theta+\lambda)}
$$
Note that since $\lambda>0$, $\sinh\theta+\lambda$ never vanishes whenever $\theta\geq0$. By naming $\theta_0<0$ to the unique solution to the equation $\sinh\theta+\lambda=0$, we derive that $\Gamma(\theta)$ is only defined for $\theta<\theta_0$, has an asymptote at $\theta=\theta_0$ as $\theta\rightarrow\theta_0$ and converges to the line $z=1/2$ as $\theta\rightarrow-\infty$. See Fig. \ref{fig5}, left, the curve in green.
\begin{figure}[h]
\begin{center}
\includegraphics[width=.45\textwidth]{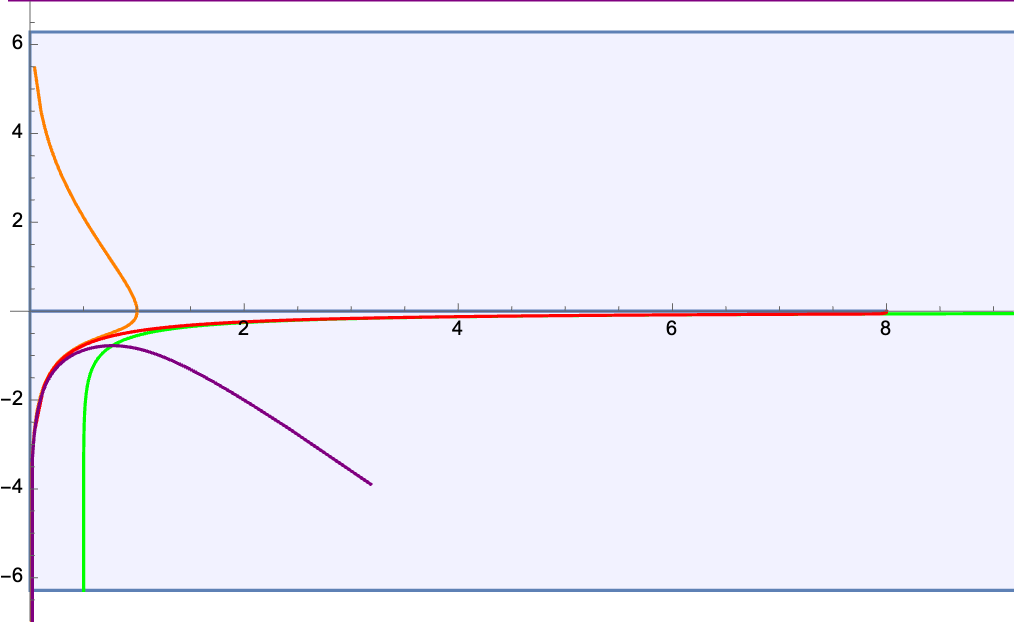}
\includegraphics[width=.45\textwidth]{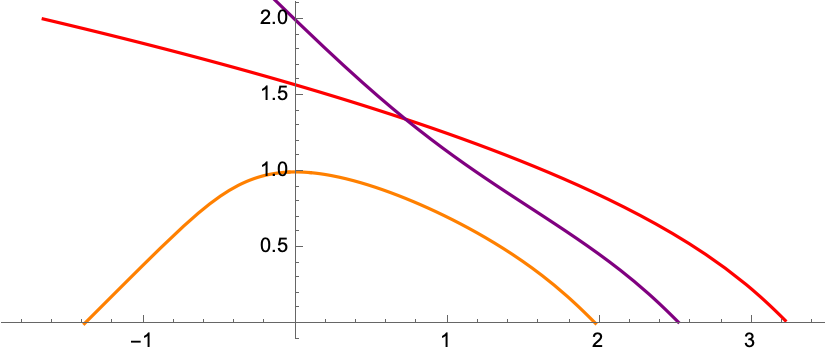} 
\end{center}
\caption{Left: the phase plane of a spacelike rotational $\lambda$-translator about a spacelike axis, and the possible orbits. Right: the profile curves of the rotational $\lambda$-translators corresponding to each orbit. Here, $\lambda=0$.}
\label{fig5}
\end{figure}

Now, the behavior of the orbits follows easily from this structure. First, fix some $z_0>0$ and let $\gamma$ be the orbit passing through $(x_0,0)$ at $s=0$. When $s$ increases, $\gamma$ lies in the region $\{\theta<0\}$ and by monotonicity never intersects $\Gamma$, hence converges to some $(0,\theta_{z_0}^1)$. When $s$ decreases, $\gamma$ lies in the region $\{\theta>0\}$ and converges to some $(0,\theta_{z_0}^2)$. The corresponding rotational $\lambda$-translator has strictly positive Gauss curvature and two cusp points at the rotation axis, starting with strictly increasing height function, reaching a maximum (precisely with value $z_0$) and then decreasing. See Fig. \ref{fig5}, the orbit and curve in orange.

A similar argument to the ones exhibited in the timelike case and the case $\lambda>1$ in the spacelike case allows us to conclude the existence of an orbit that on the one hand converges to some $(0,\theta_\infty),\ \theta_\infty<0$, and on the other hand converges to the line $\theta=\theta_0$ lying above the curve $\Gamma$ and never intersecting it. Details are also skipped at this point. The corresponding rotational $\lambda$-translator has strictly decreasing height function and reaches the rotation axis with a cusp point. Moreover, its Gauss curvature is strictly positive. See Fig. \ref{fig5}, the orbit and curve in red.

Finally, take some $(z_1,\theta_1)\in\Gamma$ and let $\gamma$ the orbit passing through this point at $s=0$. When $s$ increases, $\gamma$ lies in the region $\{\theta<0,\ z<\Gamma(\theta)\}$ and converges to some $(0,\theta_{z_1})$. When $s$ increases, $\gamma$ stays in the region $\{\theta<0,\ z>\Gamma(\theta)\}$. The corresponding rotational $\lambda$-translator has strictly decreasing height function, reaches the rotation axis at a cusp point and has Gauss curvature of changing sign: first negative and then positive. See Fig. \ref{fig5}, the orbit and curve in purple.

\subsubsection{The timelike case}
Finally, we approach the final case of our study by assuming $\Sigma$ timelike. Therefore, $N$ is spacelike which yields $z'=\cosh\theta,x'=\sinh\theta$, and the associated nonlinear autonomous system is
\begin{equation}
\left\lbrace
\begin{array}{l}
z'=\cosh\theta\\
\theta'=-2(\cosh\theta+\lambda)-\displaystyle{\frac{\sinh\theta}{z}}.
\end{array}\right.
\end{equation}
Note that in this case, the height function of every rotational $\lambda$-translator is strictly increasing, being the $x$-function the one susceptible of changing its monotonicity.
The graph
$$
z=\Gamma(\theta)=-\frac{\sinh\theta}{2(\cosh\theta+\lambda)}
$$
is only defined for $\theta<0$, it has no asymptote in $\theta$ and converges to $z=1/2$ as $\theta\rightarrow-\infty$. See Fig. \ref{fig6}, left, the curve in green.
\begin{figure}[h]
\begin{center}
\includegraphics[width=.45\textwidth]{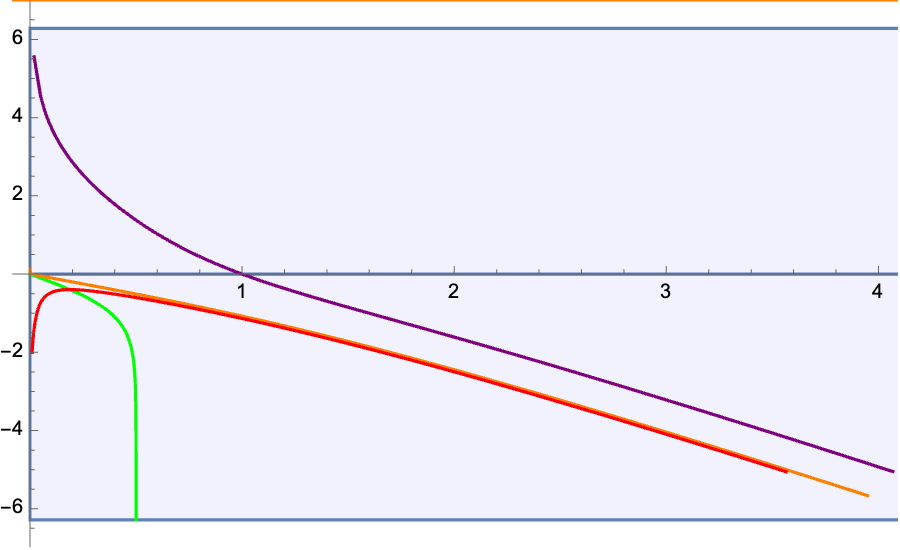}
\includegraphics[width=.35\textwidth]{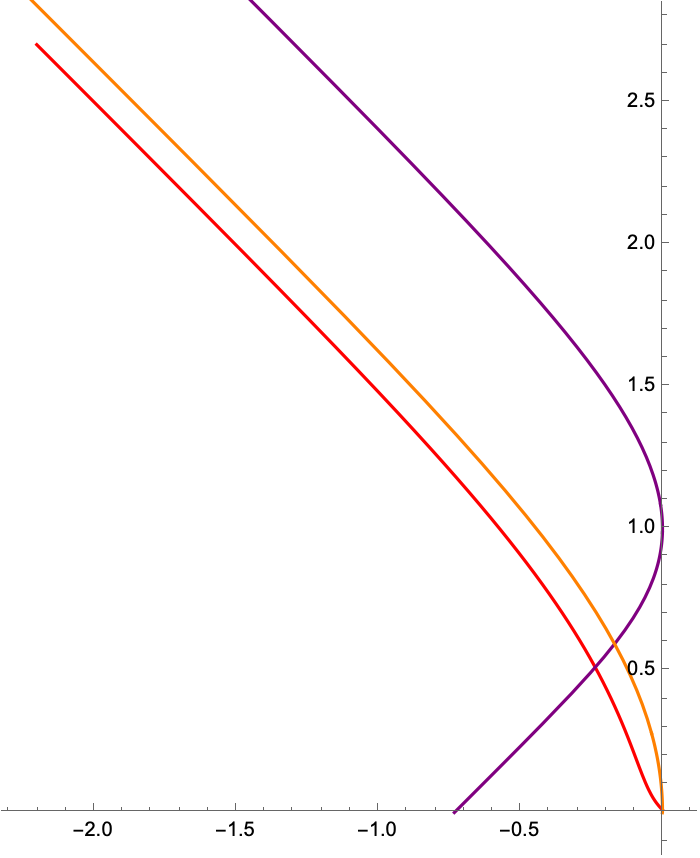} 
\end{center}
\caption{Left: the phase plane of a timelike rotational $\lambda$-translator about a spacelike axis, and the possible orbits. Right: the profile curves of the rotational $\lambda$-translators corresponding to each orbit. Here, $\lambda=0$.}
\label{fig6}
\end{figure}

First, recall the existence of a rotational $\lambda$-translator about the $e_1$-axis and intersecting it orthogonally, see Th. \ref{thexistenciaradial}. The corresponding orbit, $\gamma$, has $(0,0)$ as endpoint and by comparison $x'<0$, which yields $\theta<0$, and therefore it lies in the region $\{\theta<0,\ z>\Gamma(\theta)\}$. The corresponding $\lambda$-translator has strictly decreasing $x$-function and strictly positive Gauss curvature. See Fig. \ref{fig6}, the orbit and curve in orange.

Now, every orbit $\gamma$ remaining either passes through some $(z_0,0),\ z_0>0$, or intersects $\Gamma$ at some $(z_1,\theta_1)$, both assumed at the instant $s=0$. In the first case, when $s$ decreaes, $\gamma$ has some $(0,\theta_{z_0})$ as endpoint in the region $\{\theta>0\}$, and when $s$ increases it lies in the region $\{\theta<0\}$. The corresponding rotational $\lambda$-translator has $x$-function starting being increasing and then decreases, and has Gauss curvature of changing sign: first negative and then positive. See Fig. \ref{fig6}, the orbit and curve in purple.

In the second case, $\gamma$ is fully contained in the region $\{\theta<0\}$. When $s$ increases, the coordinates of $\gamma$ satisfy $z\rightarrow\infty$ and $\theta\rightarrow-\infty$. When $s$ decreases, $\gamma$ converges to some $(0,\theta_2)$. The corresponding rotational $\lambda$-translator has strictly decreasing $x$-function and Gauss curvature of changing sign: first negative and then positive.

This concludes the classification of every rotational $\lambda$-translator.

\section*{Statements and Declarations}  

\textbf{Funding} Antonio Bueno has been partially supported by the Project CARM, Programa Regional de Fomento de la Investigaci\'{o}n, Fundaci\'{o}n S\'{e}neca-Agencia de Ciencia y Tecnolog\'{\i}a Regi\'{o}n de Murcia, reference 21937/PI/22

Irene Ortiz is partially supported by the grant PID2021-124157NB-I00 funded by MCIN/AEI/

\vspace{-.25cm}\hspace{-.75cm} 10.13039/501100011033/ "ERDF A way of making Europe", 
Spain, and she is also supported by Comunidad Aut\'{o}noma de la Regi\'{o}n
de Murcia, Spain, within the framework of the Regional Programme
in Promotion of the Scientific and Technical Research (Action Plan 2022),
by Fundaci\'{o}n S\'{e}neca, Regional Agency of Science and Technology,
REF. 21899/PI/22.

\textbf{Authors contributions.} All authors contributed, read and approved the final manuscript.

\textbf{Conflict of interest.} The authors have no relevant financial or non-financial interests to disclose.

Departamento de Ciencias, Centro Universitario de la Defensa de San Javier, Santiago de la Ribera E-30729, Spain

$\mathtt{antonio.bueno@cud.upct.es}$, \qquad $\mathtt{irene.ortiz@cud.upct.es}$


\begin{thebibliography}{00}

\bibitem{ALO} D. Artacho, M. A. Lawn, M. Ortega, Translating solitons in generalised Robertson-Walker geometries, arXiv:2211.14529.

\bibitem{Bue} A. Bueno, Translating solitons of the mean curvature flow in the space $\mathbb{H}^2\times\mathbb{R}$, \emph{J. Geom.} \textbf{109} (2018).	

\bibitem{BuLo} A. Bueno, R. López, Compact surfaces with boundary with prescribed mean curvature depending on the Gauss map, \emph{Ann. Global Anal. Geom.} \textbf{64} (2023).

\bibitem{BLO} A. Bueno, R. López, I. Ortiz, The Plateau-Rayleigh instability of translating $\lambda$-solitons, \emph{Results Math.} \textbf{79} (2024).

\bibitem{BuOr} A. Bueno, I. Ortiz, Invariant hypersurfaces with linear prescribed mean curvature, \emph{J. Math. Anal. Appl.} \textbf{487} (2020), 124033.

\bibitem{CSS} J. Clutterbuck, O. Schnurer, and F. Schulze, Stability of translating solutions to mean curvature flow, \emph{Calc. Var. Partial Differential Equations} \textbf{29} (2007), no. 3, 281--293.

\bibitem{Gro} M. Gromov, Isoperimetry of waists and concentration of maps, \emph{Geom. Funct. Anal.} \textbf{13} (2003), 178--215.

\bibitem{Hui} G. Huisken, The volume preserving mean curvature flow, \emph{J. Reine Angew. Math.} \textbf{382} (1987), 35--48.

\bibitem{HuSi} G. Huisken, C. Sinestrari, Convexity estimates for mean curvature flow and singularities of mean convex surfaces, \emph{Acta Mathematica} \textbf{183} (1993), no. 1, 45--70.

\bibitem{Ilm} T. Ilmanen, Elliptic regularization and partial regularity for motion by mean curvature, \emph{Mem. Amer. Math. Soc.} \textbf{108} (1994), no. 520.

\bibitem{Kim} D. Kim, Rotationally symmetric spacelike translating solitons for the mean curvature flow in Minkowski space, \emph{J. Math. Anal. Appl.} \textbf{488} (2020).

\bibitem{LiMa} J. H. S. de Lira, F. Martin, Translating solitons in Riemannian products, \emph{J. Differential Equations} \textbf{266} (2019), 7780--7812.

\bibitem{LaOr} M. A. Lawn, M. Ortega, Translating Solitons in a Lorentzian Setting, Submersions and Cohomogeneity One Actions, \emph{Mediterr. J. Math.} \textbf{19} (2022).

\bibitem{Lop1} R. López, Invariant surfaces in Euclidean space with a log-linear density, \emph{Adv. Math.} \textbf{339} (2018), 285--309.

\bibitem{Lop2} R. López, Compact $\lambda$-solitons with boundary, \emph{Mediterr. J. Math.} \textbf{15} (2018).

\bibitem{MSHS} F. Martín, A. Savas-Halilaj, K. Smoczyk, On the topology of translating solitons of the mean curvature flow, \emph{Calc. Var. Partial Differential Equations} \textbf{54} (2015), no. 3, 2853-2882.

\bibitem{Pip1} G. Pipoli, Invariant translators of the Heisenberg group, \emph{J. Geom. Anal.} \textbf{31} (2021), 5219--5258.

\bibitem{Pip2} G. Pipoli, Invariant translators of the Solvable group, \emph{Ann. Mat. Pura. Appl.} \textbf{199} (2020), 1961--1978.

\bibitem{SpXi} J. Spruck, L. Xiao, Complete translating solitons to the mean curvature flow in $\r^3$ with nonnegative mean curvature, \emph{Amer. J. Math.} \textbf{142} (2020), 993--1015.
\end{thebibliography}
\end{document}